\documentclass[10 pt]{amsart}
\usepackage{amsmath}
\usepackage{mathtools}
\usepackage{amssymb}
\usepackage{amsthm}
\usepackage{fancyhdr}
\usepackage[all]{xy}
\usepackage{hyperref}
\usepackage{paralist}
\usepackage{tikz}
\usepackage{algorithm}
\usepackage{algorithmic}
\usepackage{graphicx}

\newtheorem{theorem}{Theorem}[section]
\newtheorem{lemma}[theorem]{Lemma}
\theoremstyle{proposition}

\theoremstyle{corollary}

\theoremstyle{definition}

\theoremstyle{remark}
\newtheorem{remark}[theorem]{Remark}

\numberwithin{equation}{section}

\newcommand{\F}{\mathbb F}

\begin{document}

\email{nari15@itu.edu.tr, ozdemiren@itu.edu.tr}
\title{Group Operation on Nodal Curves}
\author{Kubra Nari}
\author{Enver Ozdemir}
\address{Informatics Institute, Istanbul Technical University}

\begin{abstract}

 In this work, we present an efficient method for computing in the Generalized Jacobian of special singular curves. The efficiency of the operation is due to representation of an element in the Jacobian group by a single polynomial.

\end{abstract}
\thanks{Classification: 11G99, 11G20, 11Y99}
\maketitle

\specialsection*{\bf \Large Introduction}
\indent The Jacobian groups of smooth curves, especially for those belonging to the elliptic and hyperelliptic curves, have been rigorously investigated \cite{Anderson, Bauer,Cantor, Makdisi, Mumford} due to  their use in computational number theory and cryptography \cite{CohFrey, HLEN, KOBEL, KOBHYP, VMIL}. Even though the singular counterpart of these curves have simple geometric structures, the generalized Jacobian groups of these curves might be potential candidates for further applications in computational number theory or cryptography.

\indent An element in the Jacobian of a hyperelliptic curve is represented by a pair of polynomials $(u(x),v(x))$ satisfying certain conditions \cite{Mumford}. The situation is the same for higher degree curves. For example, an element $D$ in the Jacobian of a superelliptic curve $S:y^3=g(x)$ is represented by  a triple  of  polynomials $(s_1(x),s_2(x),s_3(x))$ satisfying certain conditions \cite{Bauer}. Therefore, for a smooth curve, we do not have the liberty to choose any polynomial $u(x)$ and say that it is a coordinate of an element in the Jacobian group of the curve.  On the other hand, the results of this work will allow us to treat almost any polynomial $h(x)$ as an element of the Jacobian of a nodal curve. We believe that this might encourage researchers to work with these curves for further applications in the related areas. \\
\indent In this work, we present an efficient method to perform group operation on the Jacobians of nodal curves. The method is basically a modification of Mumford representation\cite{Mumford} and Cantor's algorithm\cite{Cantor}. We note  that in the work \cite{OzdemirPHD}, Mumford representation and Cantor's algorithm are extended for  general singular curves. For our purpose, a nodal curve $N$ over a finite field $\F_q$ with a characteristic $p\ne 2$ is a curve defined by an equation $y^2=xf(x)^2$ where $f(x)\in \F_q[x]$ is an irreducible polynomial. Let $d=\deg(f(x))$. We  show that almost any polynomial $h(x)$ with $\deg(h(x))<d$ uniquely represents an element $D$ in the Jacobian of the curve. Then, we define an addition algorithm for this  single polynomial representation in the Jacobian group. The representation provides great advantageous and the implementation results are illustrated by the end of the paper.

\section{Singular  Curves}
\indent  The Jacobian is an abstract term which attaches an abelian group to an algebraic curve. This abstract group, Jacobian,  is simply the ideal class  group of the corresponding coordinate ring.  If the 
curve is smooth, the attached group is called  Jacobian, otherwise it is called  Generalized Jacobian \cite{Rosenlicht}. However, we will keep using the term `Jacobian' for all kinds of curves. We are only interested in computing in the Jacobian groups of nodal curves and more details about algebraic and geometric properties of these curves can be found in \cite{Bosch,LIU}. As we mentioned above, for our purposes, a nodal curve is defined by an equation of the form $N:y^2=xf(x)^2$ over a field $\mathbb F_q$ with a characteristic different from 2 where $f(x)$ is an irreducible polynomial in $\mathbb F_q[x]$. The attached  Jacobian group is denoted by Jac($N$). For example, if the degree of $f(x)$ is 1, that is $N:y^2=x(x+a)^2$ for some $a\ne 0\in \F_q$,  computing in the Jacobian group is similar to computing in an elliptic curve group \cite[Section 2.10]{WAS}. In order to perform group operation for a curve, each element in the Jacobian should be represented in a concrete way. The Mumford representation provides a concrete representation for elements in the Jacobians of hyperelliptic curves. This representation has been extended\cite{OzdemirPHD} for singular curves defined by equations of the form $y^2=g(x)$. Below, we present Mumford representation along with Cantor's algorithm which provides a method of computing in the Jacobians for aforementioned singular curves \cite{OzdemirPHD}.
\subsubsection{The Mumford Representation}
\indent Let $f(x) \in \mathbb F_q$ be a monic  polynomial of degree $2g+1$  such that $g\ge 1$. A curve $H$ over $\F_q$ is defined by  the equation $y^2=f(x)$.  Any divisor class $D$ in the Jacobian group of $H$, Jac($H$), is represented by a pair of polynomials $[u(x),v(x)]$ satisfying the following:

\begin{enumerate}
\item $\deg(v(x))<\deg(u(x))$.
\item $v(x)^2-f(x)$ is divisible by $u(x)$. 
\item If $u(x)$ and $v(x)$ are both multiples of $(x-a)$ for a singular point $(a,0)$ then
$\dfrac{f(x)-v(x)^2}{u(x)}$ is not a multiple of $(x-a)$. Note that $(a,0)$ is a singular point of $H$ if $a$ is a multiple root of $f(x)$.\\
\end{enumerate}

Any divisor class $D\in$ Jac($H$) is uniquely represented by a (reduced) pair $(u(x),v(x))$ if in addition to the above properties, we have:
\begin{enumerate}
\item $u(x)$ is monic.
\item deg$(v(x))<$deg$(u(x))\leq g.$
\end{enumerate}
 We should note here that the identity element is represented by $[1,0]$.
\subsubsection{Cantor's Algorithm} 
This algorithm takes two divisor classes $D_1=[u_1(x),v_1(x)]$ and $D_2=[u_2(x),v_2(x)]$ on $H: y^2=f(x)$ and outputs the unique representative for the divisor class $D$ such that $D=D_1+D_2$. 
\begin {enumerate}
\item $h= \gcd (u_1,u_2,v_1+v_2)$ with polynomials $h_1, h_2,h_3$ such that\\ $h=h_1u_1+h_2u_2+h_3(v_1+v_2)$\\
\item $u=\dfrac {u_1u_2}{h^2}$ and $v\equiv \dfrac{h_1u_1v_2+h_2u_2v_1+h_3(v_1v_2+f)}{h}$ (mod $u$)\\
\\
{\bf repeat:}
\\
\item $\widetilde{u}=\dfrac{v^2-f}{u}$ and $\widetilde{v}\equiv v$ (mod $\widetilde u$)\\

\item $u=\widetilde u$ and $v=-\widetilde v$
\\
{\bf until} deg $(u)\leq g$ 

\item Multiply $u$ by a constant to make $u$ monic.
\item $D=[u(x),v(x)]$
\end{enumerate}
The combination of the third and fourth steps is called the reduction steps which eventually return a unique reduced divisor for each class. The justification of the above statements is given in the \cite{OzdemirPHD}.\\

\section{Nodal Curves}
A nodal curve  $N$ over a field is an algebraic curve with finitely many singular points which are all simple double points. The curve $N$ has a smooth resolution $\widetilde{N}$ obtained by separating the two branches at each node. In this section, we are  going to construct a  representation for elements in the Jacobians of  nodal curves. Again, note that the curves under consideration are of the form $N:y^2=xf(x)^2$ where $f(x)$ is an irreducible polynomial of degree $d$ over the field $\mathbb F_q$. Here, we briefly mention related results for nodal curves especially from the work of M. Rosenlicht \cite{Rosenlicht, Rosenlicht54}. Let 
$C$ be a smooth algebraic curve and $\mathfrak m$ be a modulus, i.e. $\mathfrak m=\sum_{P\in C} m_PP$ where $m_P$ is non-negative. Let denote the generalized Jacobian group of $C$ with respect to the modulus $\mathfrak m$ by $J_{\mathfrak m}(C)$. We have a surjective homomorphism\cite{Rosenlicht54} 
$$\sigma : J_{\mathfrak m}(C)\rightarrow Jac(C).$$  
\begin{remark}\label{Rm1} The normalization of the nodal curve $N$ gives $\mathbb P^1$ so we take $C=\mathbb P^1$. It is known that Jac($C)$ is trivial. In our case i.e., $J_{\mathfrak m}(C)=$Jac($N)$, the kernel of $\sigma$ is isomorphic to a torus $\mathbb G_m^d$ of dimension $d$=deg($f(x)$). Note that, the modulus $\mathfrak m$ has only singular points which are the roots of $f(x)$. See  \cite{Dechene,Serre} for more details. 
\end{remark}	

\begin{theorem} \label{theorem1}
Let $f(x)$ be an irreducible polynomial of degree $d$ over $\F_q$ and $N:y^2=xf^2(x)$ be a nodal curve over $\mathbb F_q$. Any divisor class $D\in$ Jac($N$) is uniquely represented by a  polynomial $h(x)$ satisfying $$\deg(h(x))< d \text{ and }  \gcd(f(x),x-h^2(x))=1.$$  
\end{theorem}

We are going to prove Theorem \ref{theorem1} by a series of lemma.

\begin{lemma} \label{Lm1}
Let $N$ be as above. Let $h(x)$ be a polynomial of degree less than $d$ such that $\gcd(f(x),x-h^2(x))=1$. Then, the pair $D=[f^2(x),h(x)f(x)]$ represents an element in Jac($D$).
\end{lemma}

\begin{proof}
Let $$D=[u(x),v(x)]=[f^2(x),h(x)f(x)].$$ Both $u(x)$ and $v(x)$ are divisible by $x-a$ where $a$ is any root of $f(x)$ over the algebraic closure of $\mathbb F_q$. On the other hand, $\gcd(f(x),x-h^2(x))=1$ so $$\dfrac{xf^2(x)-v^2(x)}{u(x)}=\dfrac{xf^2(x)-h^2(x)f^2(x)}{f^2(x)}=x-h^2(x)$$ is not divisible by $x-a$ for any root $a$ of $f(x)$. By Mumford Representation which is described above, $[f^2(x),h(x)f(x)]$ represents an element $D$ in Jac($N$).
\end{proof}

\begin{lemma}\label{Lm2}
Let $\gcd(f(x),x-h_i^2(x))=1$ and $\deg(h_i(x))<d$ for each $i=1,2$. Let $D_1=[f^2(x),h_1(x)f(x)]$ and $D_2=[f(x)^2,h_2(x)f(x)]$ be two divisor classes.
We find $$D_1+D_2=D_3=[f^2(x),h_3(x)f(x)]$$ via
\begin{enumerate}
\item finding two polynomials $g_1(x),g_2(x)$ such that $$g_1(x)f(x)+g_2(x)(h_1(x)+h_2(x))=1\\$$
\item Then computing $$h_3(x)\equiv (f(x)h_1(x)g_1(x)+g_2(x)(h_1(x)h_2(x)+x)) \mod f(x)$$ with $\deg(h_3(x))<d$.\\
\end{enumerate} 
\end{lemma}

\begin{proof}
We apply Cantor's Algorithm for $D_1+D_2$ to confirm the addition algorithm. 
\begin{enumerate}
\item We first compute:\\

$\begin{array}{lll}
       \gcd(f(x)^2,f(x)^2,h_1(x)f(x)+h_2(x)f(x))&=& f(x)\cdot \gcd(f(x),f(x),h_1(x)+h_2(x))\\
       \\
       &=& f(x)\cdot \gcd(f(x),h_1(x)+h_2(x))\\
       \\
       &=&f(x)
      \end{array}$
 
with $g_1(x),g_2(x)$ such that $g_1(x)f(x)+g_2(x)\Big(h_1(x)+h_2(x)\Big)=1.$  \\

\item Set\\
$\begin{array}{lllll}
u_3(x)&=&\dfrac{f(x)^2f(x)^2}{f(x)^2}=f(x)^2 \\
\\
v_0(x)&=&\dfrac{g_1(x)f^3(x)h_1(x)+g_2(x)\Big(h_1(x)h_2(x)f^2(x)+xf(x)^2\Big)}{f(x)}\\
\\
&&=g_1(x)f^2(x)h_1(x)+g_2(x)\Big(h_1(x)h_2(x)f(x)+xf(x)\Big).\\
\\
\end{array}$
\item Then\\
$\begin{array}{lllll}
v_3(x)&\equiv& v_0(x) \mod u_1(x)\\
\\
&\equiv& g_1(x)f(x)^2h_1(x)+g_2(x)\Big(h_1(x)h_2(x)f(x)+xf(x)\Big)\mod u_3(x)=f(x)^2\\
\\
&=&\underbrace{\Big(f(x)g_1(x)h_1(x)+g_2(x)(h_1(x)h_2(x)+x) \mod f(x)\Big)}_{h_3(x)}f(x)\\
\\
&=& h_3(x)f(x) \text{ with} \deg(h_3(x))<d\\
\\
\end{array}
$
\item $D_1+D_2=[u_3(x),v_3(x)]=[f^2(x),h_3(x)f(x)]=D_3$\\
\end{enumerate}

\end{proof}

Note that $$[f^2(x), h(x)f(x)]+[f^2(x),-h(x)f(x)]=[1,0].$$

\begin{lemma}
Let $N:y^2=xf^2(x)$ be a nodal curve over $\mathbb F_q$ such that $f(x)$ is an irreducible polynomial. Let $$\begin{array}{ccc}
D_1 & = & [f^2(x),\quad h_1(x)f(x)] \text{ with } \deg (h_1(x))<\deg(f(x))\\
D_2 &= & [f^2(x),\quad h_2(x)f(x)] \text{ with } \deg(h_2(x))<\deg(f(x))
\end{array}$$ 
such that $$h_1(x)\ne h_2(x).$$
Then $$D_1\neq D_2.$$ 
\end{lemma}

\begin{proof}
Suppose $$D_1=D_2$$ then 
$$\begin{array}{ccl}
[1,0]&=& D_1+(-D_2)\\
&=& [f^2(x),h_1(x)f(x)]+[f^2(x),-h_2(x)f(x)]\\
\end{array}$$
This is possible only when $h_1(x)+(-h_2(x))$ is zero or a multiple of $f(x)$. Note that it can not be a multiple of $f(x)$ as the degrees of both $h_1(x)$ and $h_2(x)$ are less than $\deg(f(x))$. Therefore, as long as  $h_1(x)\ne h_2(x)$, we do not get $D_1=D_2$.
\end{proof}

\noindent{\it Proof of Theorem \ref{theorem1}}:\\
In Lemma \ref{Lm1}, we defined a new type of a representation for  elements in the Jacobian group of $N:y^2=xf^2(x)$, i.e., each element is represented by a pair $[f^2(x),h(x)f(x)]$ such that $\deg(h(x))<\deg(f(x))$ and $f(x)$ doesn't divide $x-h^2(x)$. The lemma \ref{Lm2} shows how to perform the group operation with this representation. In the last lemma, we showed that for distinct $h(x)$, the pairs represent distinct elements in the Jacobian group. As the degree of $h(x)$ is less than $d$, we have approximately $q^{\deg(f(x))}$ such pairs which is equal to the order of the Jacobian group by the remark \ref{Rm1} and this completes the proof.  \\

 Let $\mathbb F_q$ be a finite field with a characteristic $p\ne2$. Let $N:y^2=xf^2(x)$ be a singular curve such that $f(x)$ is an irreducible polynomial of degree $d$ over $\mathbb F_q$. The above discussion leads us to the following algorithm.
 
\begin{algorithm}[!h]                  
	\caption{Addition algorithm for the Jacobian groups of the curves $N:y^2=xf^2(x)$ over $\mathbb F_q$. }           
	\label{alg1}                     
	\begin{algorithmic}[1]   
		\renewcommand{\algorithmicrequire}{\textbf{Input:}}
		\renewcommand{\algorithmicensure}{\textbf{Output:}}
		\REQUIRE $D_1=h_1(x) \hspace{1ex} \text{and} \hspace{1ex} D_2=h_2(x)$ where the degrees of $h_1(x)$ and $h_2(x)$ are less than the degree of $f(x)$. Note that, we represent the pair $[f(x)^2,h(x)f(x)]$ by $h(x)$.
		\ENSURE $D= D_1 + D_2=h(x)$
		\STATE If $h_1(x)+h_2(x)=0$ set $h(x)=[1,0]$ (identity). Otherwise do:
		\STATE Find $g_1(x)$ and $g_2(x)$ such that
		$g_1(x)f(x)+g_2(x)(h_1(x)+h_2(x))=1$.
		\STATE Set: $$h(x)\equiv (g_2(x)(h_1(x)h_2(x)+x)) \mod f(x)$$
		
		\RETURN $h(x)$\\
	\end{algorithmic}
\end{algorithm}

 We form the curve $N$ with an irreducible polynomial $f(x)$ of degree $d$ over $\mathbb F_q$. Any polynomial $h(x)$ of degree less than $d$ with $\gcd(f(x),x-h^2(x))=1$ represents a unique element in 
Jac($N$). For two elements $D_1,D_2\in $ Jac($N$) represented by polynomials $h_1(x)$ and $h_2(x)$ respectively, we defined an addition operation involving only univariate polynomial arithmetics. The algorithm returns a polynomial $h(x)\in\mathbb F_q[x]$ which uniquely represents $D=D_1+D_2$. We also note that almost all polynomials $h(x)$ of degree less than $d$ represents an element in Jac($N$) and this allows one easily select  a random element $D$ in Jac($N$). The single polynomial representation does not only give us a liberty to select any polynomial, it also provides an efficient group operation in the Jacobian group. The following table compares this group operation with regular Cantor's algorithm. According to the results in this table, the single representation of Jacobian elements have advantages over polynomial pairs representation.  The time is measured while computing a $pQ$ where $Q$ is an element in the Jacobian group of $N$ and $p=4294967311$. The curve is over the field $\mathbb F_p$. 
\\

\begin{tabular}{ |p{4cm}|p{3cm}|p{3.5cm}|  }

	\hline
	{\bf Degree of the $f(x)$ when $N:y^2=xf^2(x)$}& {\bf Nodal Curves (Second)} &{\bf Cantor's Algorithm (Second)}\\
	\hline
	5   & 0.003    &0.023\\
	\hline
	11&   0.01  & 0.081  \\
	\hline
	23 &0.019 & 0.357\\
	\hline
	47    &0.06 & 2.15\\
	\hline
	53&   0.068 & 2.98\\
	\hline
	63& 0.089 & 4.96   \\
	\hline
	71& 0.12  & 6.95\\
	\hline
	83   & 0.15    &10.98\\
	\hline
	95&   0.19  & 16.67  \\
	\hline
	110 &0.24 & 26.36\\
	\hline
	130   &0.33 & 45.46\\
	\hline
	145&   0.41  & 64.9\\
	\hline
	150& 0.43  & 72.3   \\
	\hline
	165& 0.52  & 100.39\\
	\hline
	193 & 0.69 &167.29\\
	\hline
	
\end{tabular}
\\
\\
The tests were run on a Linux OS computer with 8 GB RAM and a Intel Core i7- 5600 2.6 GHz processor. We use the programming language of C++ with a PARI/GP library\cite{Pari}. The following chart illustrates the results in the table.  
\newpage
\begin{figure}
	\includegraphics[width=\linewidth]{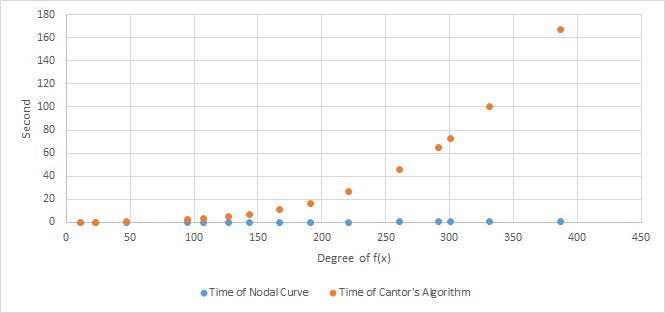}
	\caption{Nodal Curves vs Cantor's Algorithm}
	\label{fig:fig1}
\end{figure}

 \bibliographystyle{amsplain}

\begin{thebibliography}{10}
\bibitem {Anderson}  G.W. Anderson, Abeliants and their application to elementary construction of Jacobians, Advances in Mathematics 172, 169-205, 2002.
\bibitem {Bauer}  M.L. Bauer, The Arithmetic of Certain Cubic Function Fields, Math.Comp,
(73)(2003), 387-413.
\bibitem {Bosch} S. Bosch, W. Lutkebohmert, M. Raynaud, Neron Models, Springer-Verlag, 1990.  
\bibitem {Cantor} D. G. Cantor, Computing in the Jacobian of a hyperelliptic curve, Math. Comp. 48 (1987), 95-101.
\bibitem {CohFrey} H. Cohen, G. Frey, Handbook of Elliptic and Hyperelliptic Curve Cryptography, Chapman \& Hall/CRC 2005.  
\bibitem {Coh} H. Cohen, A Course in Computational Algebraic Number Theory, Springer-Verlag, 2000.
\bibitem {Dechene}  I. Dechene, Generalized Jacobians in cryptography, Ph.D. thesis, McGill University, 2005.
\bibitem{HLEN}  H. W. Lenstra Jr., "Factoring integers with elliptic curves." Annals of Mathematics (2) 126 (1987), 649-673.
\bibitem{KOBEL} N. Koblitz, Elliptic curve cryptosystems,  Mathematics of Computation 48, 1987, pp. 203–209
\bibitem{KOBHYP} N. Koblitz, Hyperelliptic cryptosystems, J. Cryptology, 1(3): 139-150, 1989
\bibitem {LIU} Q. Liu, Algebraic Geometry and Arithmetic Curves, Oxford Science Publications, 2002. 
\bibitem {Makdisi}  K. Khuri-Makdisi, Linear algebra algorithms for divisors on an algebraic curve.  Math. Comp.  73  (2004),  no. 245, 333--357
\bibitem {VMIL} V. Miller, Use of elliptic curves in cryptography, Advances in cryptology---CRYPTO 85, Springer Lecture Notes in Computer Science vol 218, 1985
\bibitem {Mumford} D. Mumford, Tata Lectures on Theta II, Birkhauser, 1982. 
\bibitem {OzdemirPHD} E. Ozdemir, Curves and Their Applications to Factoring Polynomials, Ph.D. Thesis, University of Maryland, 2009. 
\bibitem {Rosenlicht}  M. Rosenlicht, Equivalence relations on algebraic curves, Annals  of Mathematics,  56, 169-191, 1952
\bibitem{Rosenlicht54} M. Rosenlicht, Generalized Jacobian varieties, Annals of Mathematics, 59, 505-530, 1954.
\bibitem{Pari} The PARI~Group, PARI/GP version {\tt 2.11.0}, Univ. Bordeaux, 2018.
\url{http://pari.math.u-bordeaux.fr/}.
\bibitem {Serre}J. P.  Serre, Algebraic Groups and Class Fields,Springer-Verlag, 1997.
\bibitem {WAS} L. C.  Washington, Elliptic Curves: Number Theory and Cryptography, 2nd edition.  Chapman \& Hall/CRC 2008.
\end{thebibliography}

\end{document}